\documentclass[11pt,a4paper]{amsart}
\usepackage[utf8]{inputenc}
\usepackage[english]{babel}
\usepackage{amsmath}
\usepackage{amsfonts}
\usepackage{amssymb}
\usepackage{hyperref}

\theoremstyle{plain}

\newtheorem{example}{Example}
\newtheorem{lemma}{Lemma}
\newtheorem{proposition}{Proposition}

\numberwithin{equation}{section}
\numberwithin{example}{section}

\title{Remarks on anti-quasi-Sasakian manifolds}
\author{Piotr Dacko}

\begin{document}
\maketitle
\begin{abstract}
In this short note we present some remarks concerning anti-quasi-Sasakian manifolds. Some proofs of their basic properties are simplified. We also  discuss some canonical invariant distributions which exist on every anti-quasi-Sasakian manifold.
\end{abstract}
\section{Introduction}
Just recently in \cite{Dil1} was introduced a new class of almost contact metric manifolds called by the authors anti-quasi-Sasakian manifold. Let $(\phi,\xi, \eta,g)$ denote corresponding almost contact metric structure. One of remarkable property of anti-quasi-Sasakian manifold is
\begin{equation}
d\eta(\phi X,\phi Y) = - d\eta(X,Y),
\end{equation}
therefore form $d\eta$ is $\phi$-anti-invariant. From definition almost contact metric manifold is said to be anti-quasi-Sasakian, if 
\footnote{Provided definition differs from original due to 
unfortunate decision in \cite{Dil1}, to drop scalar factor 2, in coboundry formula $2d\eta(X,Y)=X\eta(Y)-Y\eta(X)-
\eta([X,Y])$}
\begin{equation}
N^{(1)} = 4d\eta\otimes\xi,\quad  d\Phi=0,
\end{equation}
where as usually $N^{(1)}$ is defined by $N^{(1)} = 
N_\phi+2d\eta\otimes \xi$, and $N_\phi$ denotes Nijenhuis torsion of $\phi$.  Therefore equivalently manifold is anti-quasi-Sasakian if  Nijenhuis torsion of $\phi$, satisfies
\begin{equation}
N_\phi=2d\eta\otimes \xi,
\end{equation}
and $d\Phi=0$.Important property of anti-quasi-Saskian manifold is that, 
in nontrivial case, $\nabla_\xi\phi \neq 0$. Note that most widely studied classes of manifolds contact metric, almost
Kenmotsu and almost cosymplectic (or almost coK\"ahler) all satisfy $\nabla_\xi\phi = 0$. So we may see anti-quasi-Sasakian manifolds as an interesting class of manifolds with $\nabla_\xi\phi \neq 0$. 

For more information about anti-quasi-Saskian manifold we refer the introductory paper \cite{Dil1}. For example 
for relations between this class of manifolds and hyperK\"ahler geometry. 

\section{Preliminaries}
Here we provide basic properties of almost contact metric
manifolds. It is assumed all manifolds are smooth, connected. Capital letters X, Y, Z,...  usually are used to denote vector fields on manifold, if not otherwise stated. Nijenhuis torsion of $\phi$, is 
\begin{equation}
N(X,Y) = \phi^2[X,Y]+[\phi X, \phi Y] -\phi [\phi X,Y] -
\phi[X,\phi Y].
\end{equation}

Let $\mathcal{M}$ be $(2n+1)$-dimensional manifold, 
$n \geqslant  1$. An almost contact metric structure on 
$\mathcal{M}$ is a quadruple of tensor fields 
$(\phi,\xi,\eta,g)$, where $\phi$ is an affinor ($(1,1)$-tensor field), $\xi$ is a vector field, $\eta$ is a 1-form, and
$g$ a Riemannian metric. Vector field $\xi$ and the form $\eta$ are called 
characteristic vector field and characteristic form. It is assumed that 

\begin{gather}
 \phi^2 X = -X +\eta(X)\xi, \quad \eta(\xi) =1 , \\
 g(\phi X, \phi Y) = g(X,Y) -\eta(X)\eta(Y). 
\end{gather}
From definition follows that tensor field $\Phi(X,Y) = g(X,\phi Y)$ is skew-symmetric, 
 a 2-form on $\mathcal{M}$, called fundamental form. 
 
 Manifold equipped with fixed almost contact metric structure is called almost contact 
 metric manifold. Almost contact metric manifold is said to be normal if $N^{(1)} =0$. 
 Among most extensively studied classes of almost contact metric manifolds are 
 contact metric manifolds, defined by $d\eta =\Phi$, almost Kenmotsu manifolds, 
 $d\eta = 0$, $d\Phi = 2 \eta\wedge\Phi$, and almost cosymplectic manifolds 
 (or almost coK\"ahler in different terminology), $d\eta=0$, $d\Phi=0$.

 Note dimension five here is exceptional as there are some exceptional five-dimensional 
 manifolds. For example manifold where characteristic form is a contact form and 
 $d\Phi = 2 \eta\wedge\Phi$, formally condition is the same as in case of almost Kenmotsu
 manifold. For dimensions $\geqslant 6$ we have that $d\Phi = 2 \eta\wedge \Phi$ always
 follows $d\eta=0$. 
 
 Almost contact metric manifold $\mathcal M$ is called anti-quasi-Sasakian if 
 \begin{equation}
   N^{(1)} = 4d\eta\otimes \xi,\quad d\Phi =0,
 \end{equation}
 ie. the fundamental form of $\mathcal{M}$ is 
 closed \cite{Dil1}.

 For general literature on the subject we refer to \cite{Bl1},  
 \cite{CPM1},\cite{CG}, \cite{Dil2}, \cite{GY}, \cite{O1}, \cite{Sal}.
 
 \section{Basic properties of anti-quasi-Sasakian manifolds}
 In this section we reprove some basic statements concerning anti-quasi-Sasakian 
 manifolds. For original proofs we refer to introductory paper \cite{Dil1}.
 
 Let us make short digression into complex geometry. If $J$ is almost 
 complex structure on manifold, its Nijenhuis torsion has simple symmetry 
 with resp. to $J$. It is very easy to verify following formula 
 \begin{equation}
 N_J(JX,JY) = -N_J(X,Y),
 \end{equation}
 in lemma below we provide similar formula in case of almost contact metric structure.
 
\begin{lemma}
Let $\mathcal{M}$ be an almost contact metric manifold. Then Nijenhuis torsion of a 
structure tensor $\phi$, satisfies
\begin{align}
& N_\phi(\phi X, \phi Y) = -N_\phi(X,Y) - 2d\eta(X,Y)\xi-2d\eta(\phi X,\phi Y)\xi+ \\ 
& \qquad       2(\eta\wedge\mathcal{L}_\xi \eta)(X,Y))\xi - 
\eta(X)\phi (\mathcal{L}_\xi\phi)Y +\eta(Y)\phi(\mathcal{L}_\xi\phi)X, 
\nonumber
\end{align}
\end{lemma} 
\begin{proof}
Proof is very direct from definition of Nijenhuis torsion.
\end{proof}
 We emphasize that results above holds true for every almost contact metric manifold.

From above Lemma we obtain almost immediately the following statement.
\begin{proposition}
On anti-quasi-Sasakian manifold following properties are satisfied
\begin{equation}
\mathcal{L}_\xi \phi =0, \quad d\eta(\phi X,\phi Y) = -d\eta(X,Y).
\end{equation}
\end{proposition} 
\begin{proof}
Indeed. By definition for anti-quasi-Saskian manifold we have for Nijenhuis torsion $N_\phi(X,Y)=2d\eta(X,Y)\xi$. Applying now 
above Lemma we obtain two independent equations
\begin{gather}
4d\eta(X,Y)+4d\eta(\phi X,\phi Y) -2 \eta\wedge \mathcal{L}_\xi\eta(X,Y) = 0, \\
\eta(X)\phi (\mathcal{L}_\xi\phi)Y - \eta(Y)\phi(\mathcal{L}_\xi\phi)X = 0,
\end{gather}
observe that the first equation follows $\mathcal L_\xi\eta=0$, 
and therefore $d\eta(\phi X,\phi Y)=-d\eta(X,Y)$. From the second one we find,  
$\phi(\mathcal{L}_\xi\phi)=0$, which in virtue of $\mathcal{L}_\xi\eta=0$, 
implies $\mathcal{L}_\xi\phi =0$.  
\end{proof}

Let define $A = -\phi \nabla\xi$. For covariant derivative $\nabla\phi$, we have 
\begin{equation}
(\nabla_X\phi)Y = 2\eta(X)Y+\eta(Y)AX+g(X,AY)\xi.
\end{equation}
It is not difficult to prove above formula. We use some general formula for covariant 
derivative which is true for every almost contact metric manifolds, 
cf. \cite{Bl1}. For details see 
\cite{Dil1}.

Note that characteristic vector field is Killing vector field, $\mathcal{L}_\xi g=0$, 
moreover $d\eta(X,Y) = g(X,S_\xi Y)$, where $S_\xi = -\nabla\xi$. We also have 
$A\phi +\phi A =0$, \cite{Dil1}.

Before proceeding further we need some glimpse what is supposed to be canonical in 
some sense anti-quasi-Sasakian structure in terms of local description.

\begin{example}
Let $\mathcal M = \mathbb{R}^{2n+1}$, $n\geqslant 2$. Let 
$(t,x^i,y^i)$ be global coordinates of a point $p \in \mathcal{M}$. Let
$a(v,w)$ be a linear skew-symmetric form on linear vector space $\mathbb{V}^n$,
 $a(v,w)= -a(w,v)$, with coefficients $a_{ij}$, $i,j=1,\ldots n$.
Set $\eta = dt +  a_{ij}(y^jdx^i-y^idx^j)$, we assume here sum on repeating 
indices. Note if rank of $a$ is 
maximal $\eta$ is 
a contact form on $\mathcal M$. As characteristic vector field we take 
$\xi = \partial_t$. Let
\begin{equation}
X_k = \partial_{y^k},\quad X_{k+n} =-2a_{kj}y^j\partial_t+
      \partial_{x^k}, 
\end{equation}  
tensor field $\phi$ is now defined by conditions
\begin{equation}
\phi \xi = 0, \quad \phi X_i = - X_{k+n}, \quad \phi X_{i+n} = X_i, 
\quad k = 1,\ldots n, 
\end{equation} 
metric is defined by requirements that global frame $(\xi, X_i, X_{i+n})$, 
$i=1,\ldots n$ is orthonormal. Therefore we have almost contact metric structure 
$(\phi, \xi,\eta,g)$ on $\mathcal{M}$. Dual coframe is given by 
$(\eta,  dx^i, dy^i)$, $i=1,\ldots n$,
therefore metric  explicitly  is described by 
\begin{equation}
g = \eta\otimes\eta +\sum_{i=1}^n (dx^i\otimes dx^i + 
    dy^i\otimes dy^i),
\end{equation}  
and for fundamental form we find $\Phi = 2 dx^1\wedge dy^1 +\ldots dx^n\wedge dy^n$, 
hence $d\Phi =0$. Note that commutators $[X,Y]= f\xi$ for some function $f$, for 
every pair $Y$, $Z \in \{\xi, X_i, X_{i+n}, i=1,\ldots n\}$, therefore for Nijenhuis 
torsion we have $N_\phi(X,Y) = \eta([\phi X, \phi Y])\xi =- 2d\eta(\phi X,\phi Y)$. 
Finally from defintion of $\eta$ we find $d\eta(\phi X, \phi Y) = -d\eta(X,Y)$. So we 
conclude that $\mathcal{M}$ equipped with this structure became anti-quasi-Sasakian 
manifold.  
\end{example}
Now we make some rather evident remarks. First the frame in the example creates 
Lie algebra. Hence there is unique Lie group structure $\mathbb H^{2n+1}$ on $\mathcal{M}$, such that 
the vector fields and the itself become left-invariant. The group $\mathbb{H}^{2n+1}$, 
is 2-step nilpotent Lie group. By Malcev theorem if coefficients of skew form $a$ are 
rational, $\mathbb{H}^{2n+1}$, admits cocompact lattice $\Gamma$. So quotient 
$\mathbb{H}^{2n+1}/\Gamma$ is compact nilmanifold. Moreover it can be equipped with 
corresponding almost contact metric structure in a way it became compact anti-quasi-Saskian 
manifold. So example above also serve are as source of compact examples of anti-quasi-Saskian 
manifolds.  

In our example we found $N_\phi = -2d\eta(\phi X,\phi Y)\xi$ which suggest to consider 
more general class of manifolds than anti-quasi-Sasakian. Namely we require 
\begin{equation}
N^{(1)} = 2(d\eta(X,Y) -  d\eta(\phi X,\phi Y))\xi,\quad d\Phi=0,
\end{equation}
now depending on symmetries of $d\eta$ we obtain quasi-Sasakian or anti-quasi-Sasakian 
manifolds. So such class of manifolds contains both quasi-Sasakian and anti-quasi-Sasakian 
manifolds.

\section{Some curvature properties}
In this we are interested in Jacobi operator $J_\xi: X\mapsto R_{X\xi}\xi$, 
of characteristic vector field, where $R_{XY}Z=[\nabla_X,\nabla_Y]Z -
\nabla_{[X,Y]}Z$ is curvature operator of Levi-Civita connection.

Note we have identity 
\begin{equation}
R_{XY}\xi = - (\nabla_X S_\xi)Y + (\nabla_Y S_\xi)X,
\end{equation} 
 identity 
\begin{equation}
3d\omega(X,Y,Z) = (\nabla_X\omega)(Y,Z) + (\nabla_Y\omega)(Z,X) + 
                (\nabla_Z\omega)(X,Y),
\end{equation}
true for every 2-form and every free-torsion connection, applied to $d\eta$, gives
\begin{equation}
g(R_{XY}\xi, Z) = -g(X,(\nabla_Z S_\xi)Y),
\end{equation}
setting $Y=\xi$, we obtain
\begin{equation}
g(R_{X\xi}\xi,Z)=g(J_\xi X,Z) = -g(X, S^2_\xi Y), 
\end{equation}
therefore symmetric form $g(J_\xi X, Z) \geqslant 0$ is semi-definite. If $S_\xi$ has 
maximal rank, equivalently $\eta$ is contact form, Jacobi operator $J_\xi$ has maximal 
rank. In such a case manifold is irreducible as Riemannian manifold, ie. it cannot 
be decomposed to Riemann product, even locally. So we have 
\begin{proposition}
If characteristic form of anti-quasi-Sasakian manifold $\mathcal{M}$, is  
a contact form, $\mathcal{M}$ is irreducible as Riemannian manifold. 
\end{proposition}
For $Z=X$, and $X$ unit and orthogonal to $\xi$, $g(R_{X\xi}\xi,X)$ is $\xi$-sectional 
curvature in direction $X$. If $\xi$ sectional curvature is constant $=c$, then 
it must be $c > 0$, and $S_\xi^2 = -c(Id-\eta\otimes\xi)$. 

Finally we will provide canonical splitting of anti-quasi-Saskian manifold. 
We will show that there are naturally defined integrable distributions. 
Set $\mathcal{H}_1$ as a distribution, such that  its sections are vector fields, 
which satisfy $\nabla_X\phi =0$. Assume rank of $\mathcal H_1$ is 
constant. We define $\mathcal{H}_2$, as 
orthogonal compliment of $\{\xi\}\oplus\mathcal{H}_1$. Thus we have 
orthogonal splitting 
\begin{equation}
T\mathcal{M}=\{\xi\}\oplus\mathcal{H}_1\oplus\mathcal{H}_2.
\end{equation}
\begin{proposition}
We have
\begin{itemize}
\item[(1)] Distribution $\mathcal{H}_1$ is involutive, the leaves carry natural
 K\"ahler structure;
\item[(2)] Distribution $\{\xi\}\oplus\mathcal{H}_1$ is involutive, leaves carry 
natural cosymplectic structure.
\end{itemize}
\end{proposition}

\begin{proof}
We start from $\mathcal{H}_1$. The first note $\mathcal{H}_1$ is $\phi$-invariant. 
Indeed $(\nabla_{\phi X})\phi Y = g(\phi X, AY)\xi = g(\phi X,\phi S_\xi Y)\xi = 
g(X, S_\xi Y)\xi$. However by assumption $0=(\nabla_X\phi)Y = g(X,A Y)\xi = 
g(S_\xi X,\phi Y)\xi$, and as $Y$ is arbitrary $S_\xi X = \eta(S_\xi X)\xi$, however 
the last is zero as $\eta(S_\xi X) = g(\nabla_\xi\xi,X)$ and $\nabla_\xi\xi =0$. 
Therefore $S_\xi X =0$. 

To prove involutivnes we need to show that $(\nabla_{[X,Y]}\phi)Z=0$, whenever 
 $\nabla_X\phi = \nabla_Y\phi =0$. We have 
\begin{equation}
(\nabla_{[X,Y]}\phi)Z = 2\eta([X,Y])AZ +\eta(Z)A[X,Y] +g([X,Y],AZ)\xi,
\end{equation}
note $\eta([X,Y])= -2d\eta(X,Y) = -2g(X,S_\xi Y)=0$. The term $g([X,Y],AZ)=
-g([X,Y], S_\xi \phi Z)= -d\eta([X,Y], \phi Z)$. Now coboundry formula for exterior
derivative 
\begin{align}
& 0= 3d^2\eta(X,Y,\phi Z) = Xd\eta(Y,\phi Z)+Yd\eta(\phi Z,X)+(\phi Z)d\eta(X,Y) - \\
& \qquad d\eta([X,Y],\phi Z)-d\eta([Y,\phi Z],X) - d\eta([\phi Z, Y],X), \nonumber
\end{align}
yields $0=d\eta([X,Y],\phi Z)$ as over terms on the right hand  vanish. Therefore 
$g([X,Y], AZ)=0$, and due to symmetries of $A$, and as $Z$ is arbitrary, we have in 
fact $A[X,Y]=0$. Therefore $\nabla_{[X,Y]}\phi =0$. During the proof we have shown 
that in fact $d\eta |_{\mathcal{H}_1}=0$. Now let $\mathcal{F}$ be arbitrary leaf
of $\mathcal{H}_1$, maximal connected integral submanifold. Almost complex structure 
$J$ on $\mathcal{F}$ is defined by restriction $\phi|_{\mathcal{H}_1}$, metric by pullback, 
ie. $\mathcal{F}$ became Riemannian submanifold. Nijenhuis torsion of $J$ vanishes, 
as $d\eta$ is zero on $\mathcal{F}$, hence structure is complex. Finally fundamental 
form of Hermitian structure $(J,\tilde g)$ is closed as it is just pullback of the 
the fundamental form of $\mathcal M$. This ends the proof of part (1) in our Proposition.

To prove that $\{\xi\}\oplus\mathcal{H}_1$ is involutive we in fact show little more, 
namely commutator $[\xi, X]$ is section of $\mathcal{H}_1$, if $X$ is a section 
of $\mathcal H_1$. Therefore 
\begin{equation}
(\nabla_{[\xi,X]}\phi)Y = 2\eta([\xi, X])AY + \eta(Y)A[\xi,X]+g([\xi,X], AY)\xi,
\end{equation}
note $\eta([\xi,X]]) = -2 d\eta(\xi, X)=0$. Now we proceed in the same way as above 
to show that $A[\xi,X]=0$, hence $\nabla_{[\xi, X]}\phi=0$. Now as we already 
know that $\mathcal{H}_1$ is involutive, we see 
$[\xi, \mathcal{H}_1] \subset \mathcal{H}_1$, so $\{\xi\}\oplus\mathcal{H}_1$ is involutive,
and $\mathcal{H}_1$ is ideal. As $\{\xi\}\oplus \mathcal{H}_1$ is $\phi$-invariant 
and odd-dimensional every leaf $\mathcal{F}$ carries natural almost contact 
metric structure $(\tilde\phi,\xi, \tilde{\eta},\tilde g)$, where $\tilde{\eta}$, 
$\tilde{g}$ are just pullbacks of $\eta$, $g$. Forms $\tilde{\eta}$ and fundamental form
$\tilde{\Phi}$, of this induced structure are both closed, hence structure is almost cosymplectic. It is easy to see that in fact structure is cosymplectic, for example we
find that Nijenhuis torsion of $\tilde{\phi}$ vanishes. This ends the proof of part (2).
\end{proof}

Here we make some remarks considering dimension of anti-quasi-Sasakian manifold. 
Going back to the example from previous section we easy see that it is not possible to 
create example with $\eta$ a contact form in dimension seven. This comes from 
fact that a skew form on 3-dimensional vector space always has non-trivial kernel. 
But there is general fact that skew-form on odd-dimensional vector space has
non-trivial kernel. Thus in terms of our Example, if $n=2k+1$, ie. dimension 
of manifold is $4k+3$, there is no such structure with $\eta$ a contact form. 
Therefore above Proposition applies, $\mathcal{H}_1$ is nontrivial and 
$\dim \mathcal{H}_1 \geqslant 2$. Of course what suggest our Example is a general 
result which is true for every anti-quasi-Sasakian manifold. 
For detailed discussion we refer to \cite{Dil1}.

\vspace{1cm}

{\tiny author email: \href{mailto:piotrdacko@yahoo.com}{piotrdacko@yahoo.com}}

\end{document}